\documentclass[12pt,reqno]{amsart}%
\usepackage{amsfonts}
\usepackage{amsmath}
\usepackage{amssymb}
\usepackage{color}
\usepackage{graphicx}%
\providecommand{\U}[1]{\protect\rule{.1in}{.1in}}
\newtheorem{theorem}{Theorem}[section]
\theoremstyle{plain}

\newtheorem{lemma}{Lemma}[section]

\newtheorem{remark}{Remark}

\numberwithin{equation}{section}
\begin{document}
\title[Optimal dimension-dependent estimate of the discrete Riesz Transforms]{Optimal dimension-dependent $\ell^p$ and $\ell^{1,\infty}$ estimates of the discrete Riesz Transforms}

\author{Junjie Shao}
\address[Junjie Shao]{Laboratory of Mathematics and Complex Systems (Ministry of Education), School of Mathematical Sciences, Beijing Normal University, Beijing, 100875, China}
\email{jjshao@mail.bnu.edu.cn}

\author{Hanli Tang}
\address[Hanli Tang]{Laboratory of Mathematics and Complex Systems (Ministry of Education), School of Mathematical Sciences, Beijing Normal University, Beijing, 100875, China}
\email{hltang@bnu.edu.cn}

\author{Zewei Xu}
\address[Zewei Xu]{Laboratory of Mathematics and Complex Systems (Ministry of Education), School of Mathematical Sciences, Beijing Normal University, Beijing, 100875, China}
\email{zwxu@mail.bnu.edu.cn}

\keywords{}
\thanks{The second author
was partly supported by  National Natural Science Foundation of China (Grant No.12471100).}
%\date{ }

%\subjclass[2010]{42B35, 42B37}

\begin{abstract}
In this paper, we are concerned with the optimal dimension-dependent $\ell^p$ norm of the discrete Riesz Transforms $R_{\text{dis}}^{(k)}$ on $\mathbb{Z}^d$ given by the singular convolution kernel $K_k(m)=c_d m_k/|m|^{d+1}$, where $c_d=\Gamma(\frac{d+1}{2})/\pi^{(d+1)/2}$ .

We show that for fixed $1<p<\infty$, when $d\to \infty$
$$\|R_{dis}^{\left( k \right)}\|_{\ell ^p\left( \mathbb{Z}^d \right) \rightarrow \ell ^p\left( \mathbb{Z}^d \right)}=2c_d\left( 1+\frac{\left( \sqrt{2}+o\left( 1 \right) \right) d}{2^{\frac{d}{2}}} \right) .$$
The operator norm of $R_{\text{dis}}^{(k)}$ grows super-exponentially as $d\to\infty$ since $c_d\sim(\frac{d-1}{2e\pi})^{\frac{d-1}{2}}\sqrt{\frac{d-1}{\pi}}$ by Stirling's formula, which gives a negative answer to the conjecture proposed by Ba\~{n}uelos, Kim and Kwa\'{s}nicki in \cite{BKK}. The optimal dimension-dependent $\ell^{1,\infty}$ estimate of $R_{\text{dis}}^{(k)}$ is also established.

\end{abstract}
\maketitle
\section{Introduction}
The Riesz transforms $R^{(k)}$, $k=1,\cdots,d$, which play an important part in Analysis, are defined by
$$R^{(k)}f(x)=\mathrm{p.v.} ~c_d\int_{\mathbb{R}^d}\frac{y_k}{|y|^{d+1}}f(x-y)dy,$$
where $c_d=\Gamma(\frac{d+1}{2})/\pi^{(d+1)/2}$.

The Riesz transforms are fundamental examples of Calder\'{o}n-Zygmund singular integral operators. In the particular case $d=1$,
these families has only one element, the so-called Hilbert transform $H$ on $\mathbb{R}$
$$H(f)(x)=\mathrm{p.v.}\frac{1}{\pi}\int_{\mathbb{R}}\frac{f(x-y)}{y}dy.$$
In the celebrated paper \cite{R}, Riesz proved that the Hilbert transform is a bounded operator on $L^p(\mathbb{R})$ if and only if $1<p<\infty$.

The precise value of the norm of the Hilbert transform is
\begin{equation*}
    \|H\|_{L^{p}(\mathbb{R})\rightarrow L^p(\mathbb{R})} =\cot(\frac{\pi}{2p^{\ast}}),~ \text{where}~~p^{\ast}=\max\{p,p/(p-1)\},
\end{equation*}
which was shown by Gohberg and Krupnik \cite{GK} when $p$ is a power of $2$, and by Pichorides \cite{P} and Cole (unpublished) for all $1<p<\infty$.

Iwaniec and Martin \cite{IM} using the method of rotations, proved that the norm of the Riesz transforms is also $\cot(\frac{\pi}{2p^{\ast}})$, i.e.
\begin{equation}\label{bound RT}
 \|R^{(k)}\|_{L^{p}(\mathbb{R}^d)\rightarrow L^p(\mathbb{R}^d)} =\cot(\frac{\pi}{2p^{\ast}}),~~~k=1,\cdots,d.
\end{equation}
An alternative, probabilistic proof of the estimate (\ref{bound RT}), based on a sharp inequality for
orthogonal martingales, was given by Ba\~{n}uelos and Wang in \cite{BW}.

In this paper, We will investigate the operator norm of the discrete Riesz Transforms. The literature on Discrete Analogues in Harmonic Analysis, which
concerns the study of discrete counterparts on $\mathbb{Z}^d$ of singular integrals and other classical operators in Harmonic Analysis on $\mathbb{R}^d$.  We refer the interested readers to the papers \cite{CZ}, \cite{SW1}, \cite{SW2}, \cite{MSW}, \cite{Pie}, \cite{Pie1}, \cite{K} and  references therein.

The discrete Hilbert transform
$$H_{dis}f(n)=\frac{1}{\pi} \sum_{m\in \mathbb{Z}\setminus \{0\}}\frac{f(n-m)}{m}$$
was introduced by Hilbert in 1909. Riesz in \cite{R} not only proved the $L^{p}$ boundedness of $H$ on $L^{p}(\mathbb{R})$ but also showed that
 $$\|H_{dis}\|_{\ell^{p}(\mathbb{Z})\rightarrow \ell^p(\mathbb{Z})}=C\|H\|_{L^{p}(\mathbb{R})\rightarrow L^p(\mathbb{R})},~~~~\|H\|_{L^{p}(\mathbb{R})\rightarrow L^p(\mathbb{R})}\leq \|H_{dis}\|_{\ell^{p}(\mathbb{Z})\rightarrow \ell^p(\mathbb{Z})}$$
 with some constant $C>0$.

 The long-standing open problem about the operator norm of $H_{dis}$, which is initiated by an erroneous proof of
Titchmarsh in \cite{T1}, \cite{T2}, was proved by Laeng \cite{La} for $p=2^k$ or $p=\frac{2^k}{2^k-1},~~~~k=1,2,\cdots$ and finally solved by Ba\~{n}uelos and Kwa\'{s}nicki \cite{BK} for all $1<p<\infty$ until 2019. In fact, they proved
\begin{equation}\label{norm of HT}
\|H_{dis}\|_{\ell^{p}(\mathbb{Z})\rightarrow \ell^p(\mathbb{Z})}=\|H\|_{L^{p}(\mathbb{R})\rightarrow L^p(\mathbb{R})} =\cot(\frac{\pi}{2p^{\ast}}).
\end{equation}

More precisely, Ba\~{n}uelos and Kwa\'{s}nicki \cite{BK} constructed the probabilistic discrete Hilbert transform
$$T_{\mathbb{H}}f(n)=\sum_{m\in \mathbb{Z}\setminus\{0\}}\mathbb{K}_{\mathbb{H}}(m)f(n-m)$$
 with
$$\mathbb{K}_{\mathbb{H}}(m)=\frac{1}{\pi m}\left(1+\int_0^{\infty}\frac{2y^3}{(y^2+\pi^2m^2)\sinh^2(y)}dy\right),$$
and proved that for $1<p<\infty$,
$$\|T_{\mathbb{H}}\|_{\ell^p(\mathbb{Z})\rightarrow \ell^p(\mathbb{Z})}\leq \cot(\frac{\pi}{2p^{\ast}}).$$
Then they established the relationship between the discrete Hilbert transform $H_{dis}$ and the probabilistic discrete Hilbert transform by
$$H_{dis}f(n)=(T_{\mathbb{H}}\ast \mathcal{P})f(n)$$
with $\mathcal{P}$ being a probalility kernel on $\mathbb{Z}$ and $\ast$ denoting the convolution. Then
$$\|H\|_{\ell^p(\mathbb{Z})\rightarrow \ell^p(\mathbb{Z})}\leq \|T_{\mathbb{H}}\|_{l^p(\mathbb{Z})\rightarrow l^p(\mathbb{Z})}\leq \cot(\frac{\pi}{2p^{\ast}}),$$
which implies (\ref{norm of HT}).

Although the operator norms of the Hilbert transform and its discrete version are the same,  the operator norm of the discrete Riesz transforms
\begin{equation}\label{def}
R_{dis}^{(k)}f(n)=c_d\sum_{m\in \mathbb{Z}^{d}\setminus\{0\}}\frac{m_k}{|m|^{d+1}}f(n-m).
\end{equation}
is still an open problem and there are few relevant results about the problem. Recently, Ba\~{n}uelos, Kim and Kwa\'{s}nicki \cite{BKK}  proved that
$$\cot(\frac{\pi}{2p^{\ast}})\leq \|R_{dis}^{(k)}\|_{\ell^p(\mathbb{Z}^d)\rightarrow \ell^p(\mathbb{Z}^d)}\leq \cot(\frac{\pi}{2p^{\ast}})+C_d.$$
But there is no more information about the positive constant except it depends only on the dimension $d$.

At the same time, Ba\~{n}uelos, Kim and Kwa\'{s}nicki extended the probabilistic discrete Hilbert transform to higher dimensional version, which is the probabilistic discrete Riesz transforms
$$T_{\mathbb{R}^{(k)}}(f)(n) = \sum_{m \in \mathbb{Z}^d \setminus \{0\}} \mathbb{K}_{\mathbb{R}^{(k)}}(n-m) f(m),$$
where the kernel
\[
\mathbb{K}_{\mathbb{R}^{(k)}}(m) = -4 \int_{\mathbb{R}^d} \int_0^\infty \frac{1}{h(x, y)} \frac{\partial p_0}{\partial x_k}(x, y) \frac{\partial}{\partial y}(y p_m)(x, y) \, dy \, dx, \]
$p_m(x,y)=p(x-m,y)$ and $p(x,y)$ is the Poisson kernel of the upper half-space $\mathbb{R}^d\times \mathbb{R}_{+}$. They proved that  $\|T_{\mathbb{R}^{(k)}}\|_{\ell^p(\mathbb{Z}^d)\rightarrow \ell^p(\mathbb{Z}^d)}$ is still $\cot(\frac{\pi}{2p^{\ast}})$.

As the one dimensional case, Ba\~{n}uelos, Kim and Kwa\'{s}nicki
tried to establish the relationship between the probabilistic discrete Riesz transforms $T_{\mathbb{R}^{(k)}}$ and the discrete Riesz transforms $R_{dis}^{(k)}$. They proposed a question (see Question 8.7 in \cite{BKK}) that is there a probability kernel $\mathcal{P}^{(k)}$ on $\mathbb{Z}^d$ such that
$$R_{dis}^{(k)}f(n)=\sum_{m\in \mathbb{Z}^d}\mathcal{P}^{(k)}(n-m)T_{\mathbb{R}^{(k)}}f(m).$$
If the answer to this question is yes, then there holds
\begin{equation}\label{conjecture}
\|R_{dis}^{(k)}\|_{\ell^p(\mathbb{Z}^d)\rightarrow \ell^p(\mathbb{Z}^d)}= \cot(\frac{\pi}{2p^{\ast}}),
\end{equation}
which is a conjecture proposed by them (see Conjecture 6.4 in \cite{BKK}).

Unfortunately, the conjecture is not true. Furthermore, the weaker version (see Problem 6.5 in \cite{BKK})
\begin{equation}\label{question}
\|R_{dis}^{(k)}\|_{\ell^p(\mathbb{Z}^d)\rightarrow \ell^p(\mathbb{Z}^d)}\leq C_p
\end{equation}
for some $C_p$ independent of $d$ is also not true.

In fact
let
\begin{equation*}
    f(n) =
    \begin{cases}
        1,  & n=0 \\
        0, & n\neq 0
    \end{cases},
\end{equation*}
then $R_{dis}^{(k)}f(n)=c_d\frac{n_k}{|n|^{d+1}}$ when $n\neq 0$, and $R_{dis}^{(k)}f(0)=0$. Thus
$$\|R_{dis}^{(k)}\|_{\ell^p(\mathbb{Z}^d)\rightarrow \ell^p(\mathbb{Z}^d)}\geq c_d\left(\sum_{n\in{\mathbb{Z}^d\setminus \{0\}}}\frac{|n_k|^p}{|n|^{p(d+1)}}\right)^{\frac{1}{p}}\geq 2^{\frac{1}{p}}c_d.$$
Recall that $c_d=\Gamma(\frac{d+1}{2})/\pi^{(d+1)/2}$, then $c_d\sim (\frac{d-1}{2e\pi})^{\frac{d-1}{2}}\sqrt{\frac{d-1}{\pi}}$ by Stirling's formula, which implies (\ref{conjecture}) and (\ref{question}) are not true. Thus it is natural to ask how does the operator norm of $R_{dis}^{(k)}$ depend on the dimension $d$, which is the first question we try to answer in the paper.

Our second goal is to investigate the optimal dimension-dependent $\ell^{1,\infty}$ estimate of the discrete Riesz transforms $R_{dis}^{(k)}$. It is a classical result of Calder\'{o}n and Zygmund \cite{CZ} that the Riesz transforms are bounded operators from $L^{1}(\mathbb{R}^d)$ to $L^{1,\infty}(\mathbb{R}^d)$. In fact, their result implies
\begin{equation}\label{weak 1}
\|R^{(k)}f\|_{L^{1,\infty}(\mathbb{R}^d)}\leq C(d)\|f\|_{L^{1}(\mathbb{R}^d)}.
\end{equation}
The proof of (\ref{weak 1}) is based on the Calder\'{o}n-Zygmund decomposition, by which one can obtain an exponential dependence on the dimension of the constant $C(d)$. The
question whether the constant $C(d)$ can be taken to be dimension free was proposed by Stein in his 1986 Berkeley ICM lecture (\cite{St}, page 202). At present the best result
in this direction is that of Janakiraman \cite{Ja}(see also \cite{SS}), who show that (\ref{weak 1}) hold for $C(d)=c\ln d$ with some absolute constant $c>0$.

However, such dimension-dependent $\ell^{1,\infty}$ estimate for the discrete Riesz transforms $R_{dis}^{(k)}$ is not known in the literature. Denote
$$\|R_{dis}^{(k)}\|_{\ell^1(\mathbb{Z}^d)\rightarrow \ell^{1,\infty}(\mathbb{Z}^d)}=\sup_{0 \not \equiv f \in \ell^1(\mathbb{Z}^d)}\frac{\|R_{dis}^{(k)}f\|_{\ell^{1,\infty}(\mathbb{Z}^d)}}{\|f\|_{\ell^1(\mathbb{Z}^d) }}.$$
Here we can provide a lower bound for $\|R_{dis}^{(k)}\|_{\ell^1(\mathbb{Z}^d)\rightarrow \ell^{1,\infty}(\mathbb{Z}^d)}$ by choose the same test function
\begin{equation*}
    f(n) =
    \begin{cases}
        1,  & n=0 \\
        0, & n\neq 0
    \end{cases}.
\end{equation*}
Thus
\begin{align*}
 & \|R_{dis}^{(k)}f\|_{\ell^{1,\infty}(\mathbb{Z}^d)}=\sup_{\lambda>0}\sharp \left\{c_d\frac{|n_k|}{|n|^{d+1}}>\lambda \right\}\lambda \\
 & \geq \sharp \left\{c_d\frac{|n_k|}{|n|^{d+1}}>(1-\epsilon)c_d \right\}(1-\epsilon)c_d\geq 2(1-\epsilon)c_d
\end{align*}
for any $\epsilon>0$, which implies
$$\|R_{dis}^{(k)}\|_{\ell^1(\mathbb{Z}^d)\rightarrow \ell^{1,\infty}(\mathbb{Z}^d)}\geq 2c_d.$$
So unlike the continuous version we can not expect a dimensional free upper bound for $\|R_{dis}^{(k)}\|_{\ell^1(\mathbb{Z}^d)\rightarrow \ell^{1,\infty}(\mathbb{Z}^d)}$.

In this paper, we will establish the optimal dimension-dependent $\ell^p$ and $\ell^{1,\infty}$ estimates for the discrete Riesz transforms $R_{dis}^{(k)}$. First we investigate the $\ell^2$ norm of $R_{dis}^{(k)}$, where we can take advantage of the Fourier multipliers of related operators.
\vskip0.3cm
\begin{theorem}\label{l2}
Let $R_{dis}^{(k)}$ be the discrete Riesz transforms in (\ref{def}). For $d\geq 2$ and $k=1,\cdots,d$, there holds
$$
2c_d\left( 1+\frac{\sqrt{2}\left( d-1 \right)}{2^{\frac{d}{2}}}-\frac{1}{3^d}-\frac{6\left( d-1 \right)}{10^{\frac{d+1}{2}}} \right) \leq \|R_{dis}^{\left( k \right)}\|_{\ell ^2\left( \mathbb{Z}^d \right) \rightarrow \ell ^2\left( \mathbb{Z}^d \right)}\leq 2c_d\left( 1+\frac{\alpha d}{2^{\frac{d}{2}}} \right)
$$
for some absolute constant $\alpha>0$.
In particular, when $d\to \infty$
$$
\|R_{dis}^{\left( k \right)}\|_{\ell ^2\left( \mathbb{Z}^d \right) \rightarrow \ell ^2\left( \mathbb{Z}^d \right)}=2c_d\left( 1+\frac{\left( \sqrt{2}+o\left( 1 \right) \right) d}{2^{\frac{d}{2}}} \right) .
$$
\end{theorem}
\vskip0.3cm
\begin{remark}
In fact, we can also provide an explicit upper bound for $\|R_{dis}^{(k)}\|_{\ell^2(\mathbb{Z}^d)\rightarrow \ell^2(\mathbb{Z}^d)}$, see (\ref{explicite upper bound}).
\end{remark}
\vskip0.3cm
For the $\ell^{1,\infty}$ and $\ell^p$ estimate of the discrete Riesz transforms $R_{dis}^{(k)}$, we can still obtain the optimal dimension-dependent constant.
\vskip0.3cm
\begin{theorem}\label{weak l1 and lp}
For $d\geq 2$ and $k=1,\cdots,d$,
there exist a absolute constant $\beta,~\gamma>0$ such that
$$2c_d\leq \|R_{dis}^{(k)}\|_{\ell^1(\mathbb{Z}^d)\rightarrow \ell^{1,\infty}(\mathbb{Z}^d)}\leq  2c_d\left( 1+\frac{\beta d}{2^{\frac{d}{2}}} \right) $$
and
$$2c_d\left( 1+\frac{\sqrt{2}\left( d-1 \right)}{2^{\frac{d}{2}}}-\frac{1}{3^d}-\frac{6\left( d-1 \right)}{10^{\frac{d+1}{2}}} \right)\leq\|R_{dis}^{(k)}\|_{\ell^p(\mathbb{Z}^d)\rightarrow \ell^p(\mathbb{Z}^d)}\leq \cot(\frac{\pi}{2p^{\ast}})+2c_d\left( 1+\frac{\gamma d}{2^{\frac{d}{2}}} \right) .$$
In particular, when $d\to \infty$
$$2c_d\leq \|R_{dis}^{(k)}\|_{\ell^1(\mathbb{Z}^d)\rightarrow \ell^{1,\infty}(\mathbb{Z}^d)}\leq2c_d\left( 1+\frac{\left( \sqrt{2}+o\left( 1 \right) \right) d}{2^{\frac{d}{2}}} \right)$$
and
$$\|R_{dis}^{\left( k \right)}\|_{\ell ^p\left( \mathbb{Z}^d \right) \rightarrow \ell ^p\left( \mathbb{Z}^d \right)}=2c_d\left( 1+\frac{\left( \sqrt{2}+o\left( 1 \right) \right) d}{2^{\frac{d}{2}}} \right) .$$

\end{theorem}
\vskip0.3cm
\begin{remark}
We would like to point out that once we establish the optimal $\ell^2$ and $\ell^{1,\infty}$ estimates, the Marcinkiewicz interpolation theorem can not provide us the optimal $\ell^p$ norm $(2+o(1))c_d$. And one also can obtain explicit upper bounds for $\|R_{dis}^{(k)}\|_{\ell^1(\mathbb{Z}^d)\rightarrow \ell^{1,\infty}(\mathbb{Z}^d)}$ and $\|R_{dis}^{(k)}\|_{\ell^p(\mathbb{Z}^d)\rightarrow \ell^p(\mathbb{Z}^d)}$ in our proof.
\end{remark}
\vskip0.3cm
Let us give a brief overview over the main ideas of the proof of Theorems. Its basic strategy is to consider the continuous-discrete operator $\tilde{R}_{\text{dis}}^{(1)}$ related to the discrete Riesz transforms $R_{dis}^{(1)}$ ($k=1$), which has the same operator norm with $R_{dis}^{(k)}$, see Lemma \ref{c-d lemma} (it was proved in \cite{BKK}) and Lemma \ref{equ c-d}. The ``continuous-discrete operator'' approach is  a useful tool to bound the norm of the continuous version by that of its discrete version, see also \cite{La}, \cite{Pie}.

For the $\ell^2$ estimate, we consider the Fourier multiplier of $\tilde{R}_{\mathrm{dis}}^{(1)}$, which is
    \begin{equation*}
        m_1(\xi) =\mathrm{p.v.} \sum_{m \in \mathbb{Z}^d \setminus \{0\}}c_d\frac{m_1}{|m|^{d+1}}e^{-2\pi i m \cdot \xi}.
    \end{equation*}
One of the crucial observations, which allow us to estimate the lower and upper bounds for the $L^{\infty}$ norm of $m_1(\xi)$, is that we use the identity
\begin{equation}\label{eq:gamma_identity}
    \frac{1}{|x|^{d+1}} = \frac{1}{\Gamma\left(\frac{d+1}{2}\right)} \int_0^{\infty} t^{\frac{d-1}{2}} e^{-t|x|^2} \mathrm{d}t
\end{equation}
and rewrite
$$ m_1(\xi)= -i \frac{1}{\pi^{\frac{d+1}{2}}} \int_0^{\infty} t^{\frac{d-1}{2}} \sum_{m \in \mathbb{Z}^d \setminus \{0\}} m_1 e^{-t|m|^2} \sin(2\pi m \cdot \xi) \mathrm{d}t.$$
Then the Poisson summation formula will help us to estimate the kernel
$$\sum_{m \in \mathbb{Z}^d \setminus \{0\}} m_1 e^{-t|m|^2} \sin(2\pi m \cdot \xi).$$

For the $\ell^{1,\infty}$ and $\ell^{p}$ estimate, from \cite{BKK} we know that $\tilde{R}_{\text{dis}}^{(1)}$ is equal to the truncated Riesz Transform
$R_1^{(1)}$ plus a error term. The difficult is to obtain the optimal estimate of the operator norm of the error term. Fortunately, the identity (\ref{eq:gamma_identity}) again can
help us overcome the difficult. And we believe that our strategy can be adopted to investigate the operator norm of second order even higher order discrete Riesz Transforms, which we
are working on.

\medskip

This paper is organized as follows. Section 2 is devoted to proving the optimal dimension-dependent $\ell^2$ norm of the discrete Riesz transforms. In Section 3, we will establish the optimal dimension-dependent $\ell^{1,\infty}$ and $\ell^p$ estimates of the discrete Riesz transforms. In Section 4, we will provide some lemmas which will be needed in the proof of the main Theorems.

\section{The $l^2$ estimate}
In this section, we will prove the optimal dimension-dependent $\ell^2$ norm of the discrete Riesz transforms $R_{dis}^{(k)}$. Without lose of generality, we may assume $k = 1$.

Define the continuous-discrete operator $\tilde{R}_{\mathrm{dis}}^{(1)}$
$$\tilde{R}_{\text{dis}}^{(1)}(F)(x)=\sum_{n \in \mathbb{Z}^d \setminus \{0\}}K_1(n)F(x-n),F \in L^p(\mathbb{R}^d),x \in \mathbb{R}^d,$$
where $K_1(n)=c_d\frac{n_1}{|n|^{d+1}}$. Ba\~{n}uelos, Kim and Kwa\'{s}nicki \cite{BKK} established the operator norm of $R_{dis}^{(1)}$ and $\tilde{R}_{\mathrm{dis}}^{(1)}$  are the same in the following lemma (In fact, they considered a more general case).

\begin{lemma}[\cite{BKK}]\label{c-d lemma}
    For the discrete Riesz transform $R_{\mathrm{dis}}^{(1)}$ and $1 < p < \infty$, the following identity holds
    \begin{equation*}
        \|R_{\mathrm{dis}}^{(1)}\|_{\ell^p(\mathbb{Z}^d) \to \ell^p(\mathbb{Z}^d)} = \|\tilde{R}_{\mathrm{dis}}^{(1)}\|_{L^p(\mathbb{R}^d) \to L^p(\mathbb{R}^d)}.
    \end{equation*}
\end{lemma}

By this Lemma, we only need to consider the $L^2$ norm of the continuous-discrete operator $\tilde{R}_{\mathrm{dis}}^{(1)}$. Thus we can take the advantage of the Fourier multiplier of $\tilde{R}_{\mathrm{dis}}^{(1)}$. So we need the following lemma, which is well-known (see \cite{StWe}).
\begin{lemma}
Suppose that \(T\) is a convolution operator with the form
\[
Tf(n) = \sum_{m \in \mathbb{Z}^d} f(m)K(n-m).
\]
Then \(T\) is bounded on \(\ell^2(\mathbb{Z}^d)\) if and only if there exists \(m \in L^\infty([0,1]^d)\) such
\[
\hat{Tf}(\xi) = m(\xi)\hat{f}(\xi), \quad \xi \in [0,1]^d.
\]
Furthermore,
\[
\|T\|_{\ell^2(\mathbb{Z}^d) \to \ell^2(\mathbb{Z}^d)} = \|m\|_{L^\infty([0,1]^d)}.
\]
\end{lemma}
By Lemma \ref{multiplier} in Appendix we know that the Fourier multiplier of $\tilde{R}_{\mathrm{dis}}^{(1)}$ is
    \begin{equation*}
        m_1(\xi) = \mathrm{p.v.} \sum_{m \in \mathbb{Z}^d \setminus \{0\}} K_1(m)e^{-2\pi i m \cdot \xi}=\mathrm{p.v.} \sum_{m \in \mathbb{Z}^d \setminus \{0\}}c_d\frac{m_1}{|m|^{d+1}}e^{-2\pi i m \cdot \xi},~~~\xi \in [0, 1]^d.
    \end{equation*}
We want to point out that when $d=1$, the multiplier is
$$-\mathrm{p.v.}\sum\limits_{n \in \mathbb{Z} \setminus \{0\}}\frac{2i}{\pi n}\sin(2\pi n \xi),$$
which is equal to $-i\frac{\xi}{|\xi|}(1-2|\xi|)$ and it's $L^{\infty}$ norm is $1$. But for the higher dimension, it seems impossible to obtain such a clear formula for $m_1(\xi)$ and it
is not easy to estimate the $L^{\infty}$ norm.

In order to establish the upper and lower bounds of $\|m_1(\xi)\|_{L^\infty}$, we need to rewrite the form of $m_1(\xi)$.
Using the fact that $K_1(-m)=K_1(m)$ and identity (\ref{eq:gamma_identity}), we obtain
\begin{equation*}
 \begin{aligned}
    m_1(\xi) &= -i \sum_{m \in \mathbb{Z}^d \setminus \{0\}} K_1(m)\sin(2\pi m \cdot \xi).\\
    &= -i \sum_{m \in \mathbb{Z}^d \setminus \{0\}} c_d \frac{1}{\Gamma\left(\frac{d+1}{2}\right)} \left( \int_0^{\infty} t^{\frac{d-1}{2}} m_1 e^{-t|m|^2} \mathrm{d}t \right) \sin(2\pi m \cdot \xi) \\
    &= -i \frac{1}{\pi^{\frac{d+1}{2}}} \int_0^{\infty} t^{\frac{d-1}{2}} \sum_{m \in \mathbb{Z}^d \setminus \{0\}} m_1 e^{-t|m|^2} \sin(2\pi m \cdot \xi) \mathrm{d}t,
\end{aligned}
\end{equation*}
where the last equality holds because of Lemma \ref{form of m1} in Appendix.
Let us denote
\begin{equation*}
\begin{aligned}
    S_1(t, x) &= \sum_{m_1 \neq 0} m_1 e^{-tm_1^2} \sin(2\pi m_1 x), \\
    S_2(t, x) &= \sum_{m_2 \in \mathbb{Z}} e^{-tm_2^2} \cos(2\pi m_2 x),
\end{aligned}
\end{equation*}
and
$$\displaystyle S(t, \xi) = \sum_{m \in \mathbb{Z}^d \setminus \{0\}} m_1 e^{-t|m|^2} \sin(2\pi m \cdot \xi).$$
By the fact
\begin{equation*}
\begin{cases}
    e^{-t|m|^2} = e^{-t m_1^2}e^{-t m_2^2}\cdots e^{-t m_d^2}, \\
    \sin(2\pi m \cdot \xi) = \sin(2\pi m_1\xi_1 + 2\pi m_2\xi_2 + \cdots + 2\pi m_d\xi_d),
\end{cases}
\end{equation*}
and the parity properties of the sine and cosine functions, we have
\begin{equation*}
\begin{aligned}
    S(t, \xi) &= \left( \sum_{m_1 \neq 0} m_1 e^{-tm_1^2} \sin(2\pi m_1 \xi_1) \right) \left( \sum_{m_2 \in \mathbb{Z}} e^{-tm_2^2} \cos(2\pi m_2 \xi_2) \right) \\
    &\quad  \cdots  \left( \sum_{m_d \in \mathbb{Z}} e^{-tm_d^2} \cos(2\pi m_d \xi_d) \right).
\end{aligned}
\end{equation*}
Thus the Fourier multiplier $m_1(\xi)$ can be expressed as
\begin{equation*}
    m_1(\xi) = -i \frac{1}{\pi^{\frac{d+1}{2}}} \int_0^{\infty} t^{\frac{d-1}{2}} S(t, \xi) \, \mathrm{d}t.
\end{equation*}
where
\begin{equation*}
    S(t, \xi) = S_1(t, \xi_1)S_2(t, \xi_2) \cdots S_2(t, \xi_d).
\end{equation*}
\subsection{Upper bound}
 In the subsection let us establish the upper bound for $\|m_1(\xi)\|_{L^{\infty}}$ when $\xi \in [0,1]^d$.
 First we write
    \begin{equation*}
    \begin{aligned}
        |m_1(\xi)| &= \frac{1}{\pi^{\frac{d+1}{2}}} \left| \int_0^{\infty} t^{\frac{d-1}{2}} S(t, \xi) \, \mathrm{d}t \right| \\
        &\le \frac{1}{\pi^{\frac{d+1}{2}}} \left( \int_0^1 t^{\frac{d-1}{2}} |S(t, \xi)| \, \mathrm{d}t + \int_1^{\infty} t^{\frac{d-1}{2}} |S(t, \xi)| \, \mathrm{d}t \right),
    \end{aligned}
    \end{equation*}
then we will estimate the kernel $|S(t, \xi)|$ on $(0,1]$ and $[1,\infty)$ respectively.

Since
\begin{equation*}
     S_2(t, x) = \sum_{m\in\mathbb{Z}} e^{-tm^2} \cos(2\pi mx) = \sum_{m\in\mathbb{Z}} e^{-tm^2} e^{2\pi imx}.
\end{equation*}
By the Poisson summation formula (see Lemma \ref{lem:poisson_sums}), we have
\begin{equation}\label{S_2}
    S_2(t, x)  = \sqrt{\frac{\pi}{t}} \sum_{k\in\mathbb{Z}} e^{-\frac{\pi^2 |x-k|^2}{t}},
\end{equation}
which will be used for the estimation when $t\in (0,1]$.

For $0 < t \leq 1$, we define
\begin{equation*}
    \rho(x) := \mathrm{dist}(x, \mathbb{Z}) = \min(x, 1-x) \in \left[0, \frac{1}{2}\right],
\end{equation*}
and denote
\begin{equation*}
    \rho_1 = \rho(\xi_1), \quad \rho_2 = \rho(\xi_2), \quad \cdots, \quad \rho_d = \rho(\xi_d).
\end{equation*}
Let $k_0$ be the nearest integer to $x$. For $k \neq k_0$, we have
\begin{equation*}
    |x - k| \geq |k - k_0| - \rho(x) \geq \frac{1}{2}|k - k_0|.
\end{equation*}
Using the fact
$$\sum_{m=1}^{\infty}e^{-tm^2}\leq e^{-t}+\int_1^{\infty} e^{-tx^2}dx\leq e^{-t}+\int_1^{\infty}x e^{-tx^2}dx=(1+\frac{1}{2t})e^{-t}~~\text{for}~~~ t>0,$$
we can obtain when $t\in (0,1]$
\begin{align}\label{est S_2}
    \sum_{k \neq k_0} e^{-\frac{\pi^2 |x-k|^2}{t}} &\leq \sum_{\ell \neq 0} e^{-\frac{\pi^2 \ell^2}{4t}} \leq 2 \left( 1 + \frac{2t}{\pi^2} \right) e^{-\frac{\pi^2}{4t}} \leq 2 \left( 1 + \frac{2}{\pi^2} \right) e^{-\frac{\pi^2}{4t}}.
\end{align}
Then by (\ref{S_2}), (\ref{est S_2}) and the definition of $\rho(x)$, there holds
\begin{equation}\label{est of S2}
    |S_2(t, x)| \leq \sqrt{\pi} \left(3 + \frac{4}{\pi^2}\right) \frac{1}{\sqrt{t}} e^{-\frac{\pi^2\rho^2(x)}{t}} ~~~\text{for}~~~t\in (0,1].
\end{equation}

Next we will estimate $ S_1(t, x)$. By Lemma \ref{lem:poisson_sums} we have
\begin{align*}
    S_1(t, x) = \sqrt{\frac{\pi}{t}} \sum_{k\in\mathbb{Z}} \left(\frac{\pi(x-k)}{t}\right) e^{-\frac{\pi^2(x-k)^2}{t}}.
\end{align*}
Hence
\begin{equation*}
    |S_1(t, x)| \leq \frac{\pi^{3/2}}{t^{3/2}} \left| \sum_{k\in\mathbb{Z}} (x-k) e^{-\frac{\pi^2(x-k)^2}{t}} \right|.
\end{equation*}
Let $g(u)=u e^{\frac{-\pi^2u^2}{t}}$, then we can check that
\begin{equation}\label{est g prime}
|g^{\prime}(u)|\leq \left(1+\frac{2\pi^2u^2}{t}\right)e^{-\frac{\pi^2u^2}{t}}\leq 2e^{-\frac{\pi^2u^2}{2t}},
\end{equation}
where we use the fact $(1+2x)e^{-x}\leq 2e^{-\frac{x}{2}}$ for $x>0$.
Recall that $k_0$ denote the nearest integer to $x$ then we have
\begin{align*}
\sum_{k\in\mathbb{Z}} (x-k) e^{-\frac{\pi^2(x-k)^2}{t}}&= \sum_{j\in\mathbb{Z}} (\rho(x) - j)e^{-\frac{\pi^2(\rho(x) - j)^2}{t}}\\
&=\rho(x)e^{-\frac{\pi^2\rho(x)^2}{t}}+\sum_{j=1}^{\infty}\left[g(j+\rho(x))-g(j-\rho(x))\right].
\end{align*}
By the Mean Value Theorem, (\ref{est g prime}) and the fact $j-\rho(x)\geq j-\frac{1}{2}$, we get
$$|g(j+\rho(x))-g(j-\rho(x))|\leq 4\rho(x)e^{-\frac{\pi^2(j-1/2)^2}{2t}}.$$
Careful calculation implies
\begin{align*}
 \sum_{j=1}^{\infty}e^{-\frac{\pi^2(j-1/2)^2}{2t}}&=e^{-\frac{\pi^2}{8t}} \sum_{j=1}^{\infty} e^{-\frac{\pi^2(j^2-j)}{2t}}
\leq e^{-\frac{\pi^2}{8t}}(1+\sum_{j=2}^{+\infty} e^{-\pi^2(j-1)}) \\
&\leq e^{-\frac{\pi^2}{8t}}(1+\frac{e^{-\pi^2}}{1 - e^{-\pi^2}}) \leq 2e^{-\frac{\pi^2}{8t}}\leq e^{-\frac{\pi^2\rho(x)^2}{2t}}.
\end{align*}
Finally, by the above estimates we obtain that for any $t \in (0, 1]$,
\begin{equation}\label{est of S1}
|S_1(t, x)| \leq \frac{9\pi^{3/2}}{t^{3/2}} \rho(x) e^{-\frac{\pi^2\rho(x)^2}{2t}}.
\end{equation}

Now denote
\begin{equation*}
    I_{\mathrm{small}} = \frac{1}{\pi^{\frac{d+1}{2}}} \int_0^1 t^{\frac{d-1}{2}} |S(t, \xi)| \, dt.
\end{equation*}
where
\begin{equation*}
    S(t, \xi) = S_1(t, \xi_1)S_2(t, \xi_2) \cdots S_2(t, \xi_d).
\end{equation*}
By the above estimates (\ref{est of S2}) and (\ref{est of S1}) for $S_2(t, x)$ and $S_1(t, x)$ respectively, there holds
\begin{align}\label{est I small}\nonumber
    I_{\mathrm{small}} &\leq 9\pi^{\frac{1}{2}}(3+\frac{4}{\pi^2})^{d-1}\rho_1 \int_0^1 t^{-\frac{3}{2}} e^{-\frac{\pi^2(\rho_1^2 + \cdots + \rho_d^2)}{2t}} \, dt \\\nonumber
    &= 9\pi^{\frac{1}{2}}(3+\frac{4}{\pi^2})^{d-1}\frac{\rho_1}{\sqrt{\rho_1^2 + \cdots + \rho_d^2}} \int_{\rho_1^2 + \cdots + \rho_d^2}^{+\infty} t^{-\frac{1}{2}} e^{-\frac{\pi^2}{2}t} \, dt \\
    &\leq 9\sqrt{2}(3+\frac{4}{\pi^2})^{d-1}.
\end{align}

Thus we only need to estimate $S_1(t,x)$ and $S_2(t,x)$ for $t \ge 1$. Concerning $S_1(t, x)$ we find that
\begin{equation*}
\begin{aligned}
    |S_1(t, x)| &= \left| \sum_{m_1 \neq 0} m_1 e^{-tm_1^2} \sin(2\pi m_1 x) \right| \le 2e^{-t}+2 \sum_{m_1=2}^{\infty} m_1 e^{-tm_1^2}\\
    &\le 2(e^{-t} +2e^{-4t}+ \int_2^{\infty} xe^{-tx^2} \, \mathrm{d}x) = 2(e^{-t} + 2e^{-4t}+\frac{1}{2t}e^{-4t})\\
    & \le 2e^{-t}+5e^{-4t},
    \end{aligned}
\end{equation*}
when $t\geq 1$. As for $S_2(t, x)$,
\begin{equation*}
\begin{aligned}
    |S_2(t, x)| &\le 1+2e^{-t}+2\sum_{m_2=2}^{\infty} e^{-tm_2^2} \le 1+2e^{-t}+2 \sum_{m_2=2}^{\infty} m_2 e^{-tm_2^2} \\
    &\leq 1+2e^{-t}+2(e^{-4t} + \int_2^{\infty} xe^{-tx^2} \, \mathrm{d}x)\\
    &=1+2e^{-t}+2(e^{-4t} + \frac{1}{2t}e^{-4t})\le 1+2e^{-t}+3e^{-4t} ,\quad t \ge 1.
\end{aligned}
\end{equation*}
In fact, we also can obtain $ |S_2(t, x)|\leq 1+e^{-t}(2+\frac{1}{t})\le 1+3e^{-t}$ , for $t\geq 1$.

Using the fundamental inequality

\[
(1+x)^{d-1}
\le
1+(d-1)x+\frac{d^2}{2}(1+x)^{d-3} x^2,~~\text{for}~~x>0
\]
and the above estimates, we have
\begin{align*}
    S(t,\xi)&=S_1(t, \xi_1)S_2(t, \xi_2) \cdots S_2(t, \xi_d)\\
    &\le \bigl(2e^{-t}+5e^{-4t}\bigr)
\bigl(1+2e^{-t}+3e^{-4t}\bigr)^{d-1}\\
    &\le
2e^{-t}
+4(d-1)e^{-2t}
+6(d-1)e^{-5t}  \\
&
\quad+9d^2e^{-3t}
\bigl(1+3e^{-t}\bigr)^{d-3}
 +5e^{-4t}
\bigl(1+3e^{-t}\bigr)^{d-1}.
\end{align*}
Careful calculation implies that
\begin{align}\label{I large}\nonumber
    I_{\mathrm{large}} &= \frac{1}{\pi^{\frac{d+1}{2}}} \int_1^{\infty} t^{\frac{d-1}{2}} |S(t, \xi)| \, \mathrm{d}t\\
    &\le c_d\left(2+\frac{2\sqrt{2}(d-1)}{2^{\frac{d}{2}}}
+\frac{6(d-1)}{5^{\frac{d+1}{2}}}\right)+c_d R_1(d)+c_d R_2(d),
\end{align}
where
\[
 R_1(d)
:=
\frac{9d^2}{\Gamma\left(\frac{d+1}{2}\right)}
\int_1^\infty
t^{\frac{d-1}{2}}e^{-3t}
\bigl(1+3e^{-t}\bigr)^{d-3}\,dt,
\]
and
\[
 R_2(d)
:=
\frac{5}{\Gamma\left(\frac{d+1}{2}\right)}
\int_1^\infty
t^{\frac{d-1}{2}}e^{-4t}
\bigl(1+3e^{-t}\bigr)^{d-1}\,dt .
\]
It remains to estimate the two remainder terms. Splitting the integral into two parts and using the fact $(1+3e^{-t})^{d-3}\leq (e^{3e^{-t}})^{d-3}\leq e^{3de^{-t}} \leq e^{\frac{3}{d}}$ for $t \in [2\ln d, \infty)$, we have
\[
\begin{aligned}
\int_1^\infty
t^{\frac{d-1}{2}}e^{-3t}
\bigl(1+3e^{-t}\bigr)^{d-3}\,dt
&=
\int_1^{2\ln d}
t^{\frac{d-1}{2}}e^{-3t}
\bigl(1+3e^{-t}\bigr)^{d-3}\,dt   \\
&\quad
+
\int_{2\ln d}^{\infty}
t^{\frac{d-1}{2}}e^{-3t}
\bigl(1+3e^{-t}\bigr)^{d-3}\,dt  \\
&\le
(2\ln d)^{\frac{d+1}{2}}
\left(1+\frac3e\right)^d
+
e^{\frac{3}{d}}\frac{\Gamma\left(\frac{d+1}{2}\right)}
{3^{\frac{d+1}{2}}}.
\end{aligned}
\]
Similarly,
\[
\begin{aligned}
\int_1^\infty
t^{\frac{d-1}{2}}e^{-4t}
\bigl(1+3e^{-t}\bigr)^{d-1}\,dt
&\le
(2\ln d)^{\frac{d+1}{2}}
\left(1+\frac3e\right)^d
+
e^{\frac{3}{d}}\frac{\Gamma\left(\frac{d+1}{2}\right)}
{4^{\frac{d+1}{2}}}.
\end{aligned}
\]
Therefore,
\begin{align}\label{R1}
 R_1(d)
\le
\frac{9d^2(2\ln d)^{\frac{d+1}{2}}
\left(1+\frac3e\right)^d}{\Gamma\left(\frac{d+1}{2}\right)}
+
9d^2e^{\frac{3}{d}}\frac{1}
{3^{\frac{d+1}{2}}},
\end{align}

and
\begin{align}\label{R2}
R_2(d)
\le
\frac{5(2\ln d)^{\frac{d+1}{2}}
\left(1+\frac3e\right)^d}{\Gamma\left(\frac{d+1}{2}\right)}
+
5e^{\frac{3}{d}}\frac{1}
{4^{\frac{d+1}{2}}}.
\end{align}

Consequently, by(\ref{est I small}), (\ref{I large}), (\ref{R1}) and (\ref{R2}) we can obtain
\begin{align}\label{explicite upper bound}\nonumber
    ||m_1(\xi)||_{L^\infty}
&\le
c_d
\left(
2+\frac{2\sqrt2(d-1)}{2^{\frac{d}{2}}}
+\frac{6(d-1)}{5^{\frac{d+1}{2}}}+9d^2e^{\frac{3}{d}}\frac{1}
{3^{\frac{d+1}{2}}}+5e^{\frac{3}{d}}\frac{1}
{4^{\frac{d+1}{2}}}
\right)\\
&\quad +c_d\left(\frac{9d^2(2\ln d)^{\frac{d+1}{2}}
\left(1+\frac3e\right)^d}{\Gamma\left(\frac{d+1}{2}\right)}+\frac{5(2\ln d)^{\frac{d+1}{2}}
\left(1+\frac3e\right)^d}{\Gamma\left(\frac{d+1}{2}\right)}\right)+9\sqrt{2}(3+\frac{4}{\pi^2})^{d-1},
\end{align}
which implies
$$\|R_{\mathrm{dis}}^{(1)}\|_{\ell^2(\mathbb{Z}^d) \to \ell^2(\mathbb{Z}^d)} \le 2c_d\left(1+(\sqrt{2}+o(1))\frac{d}{2^{\frac{d}{2}}}\right).$$

\subsection{Lower bound}
We already know that $\sqrt{2}c_d$ is a lower bound of $ \|R_{\mathrm{dis}}^{(k)}\|_{\ell^2(\mathbb{Z}^d) \to \ell^2(\mathbb{Z}^d)}$ in Section 1. In this subsection, we provide a better one by investigating the lower bound of $\|m_1(\xi)\|_{L^{\infty}}$ for $\xi \in [0,1]^d$, which is asymptotically optimal.

Recall that
\begin{equation*}
    |m_1(\xi)| = \left| \frac{1}{\pi^{\frac{d+1}{2}}} \int_0^{\infty} t^{\frac{d-1}{2}} S(t, \xi) \, \mathrm{d}t \right|.
\end{equation*}
From the continuity of $m_1(\xi)$,we can choose $\xi_0 = \left(\frac{1}{4}, 0, 0, \dots, 0\right) \in [0, 1]^d$,  so we have
\begin{equation*}
    S(t, \xi_0) = S_1\left(t, \frac{1}{4}\right)[S_2(t, 0)]^{d-1},
\end{equation*}
where
\begin{equation*}
    S_2\left( t,0 \right) =\sum_{m\in \mathbb{Z}}{e}^{-tm^2}=1+2e^{-t}+2e^{-4t}+\cdots \ge 1+2e^{-t},\quad \text{for\,\,all\,\,}t>0.
\end{equation*}
and
\begin{equation*}
    S_1\left(t, \frac{1}{4}\right) = \sum_{m \neq 0} m e^{-tm^2} \sin\left(\frac{\pi m}{2}\right)=2 \sum_{k=0}^{\infty} (-1)^k (2k+1)e^{-t(2k+1)^2}.
\end{equation*}
By Lemma \ref{lem:poisson_sums} we know
$$S_1\left(t, \frac{1}{4}\right)=\frac{\pi}{4t} \sqrt{\frac{\pi}{t}} \sum_{k=0}^{\infty} (-1)^k (2k+1) e^{-\frac{\pi^2(2k+1)^2}{16t}}.$$
 We can check that $b_k(t)=e^{-\frac{\pi^2(2k+1)^2}{16t}}$ is strictly monotonically decreasing with respect to $k$, and $b_k(t) \to 0$ as $k \to \infty$ for $t\in (0,1]$. Similarly, we obtain
$$S_1\left(t, \frac{1}{4}\right) = \frac{\pi}{4t}\sqrt{\frac{\pi}{t}}[(b_0(t) - b_1(t)) + (b_2(t) - b_3(t)) + \cdots] > 0~~\text{for}~~~t\in (0,1].$$

For $t\geq \frac{1}{50}$, it is easy to check that $a_k(t)=(2k+1)e^{-t(2k+1)^2}$ is strictly monotonically decreasing with respect to $k~(k \ge 2)$ and $a_k(t) \to 0$ as $k \to \infty$.
Thus
\begin{equation*}
\begin{aligned}
    S_1\left(t, \frac{1}{4}\right) &> 2(a_0(t) - a_1(t)) = 2(e^{-t} - 3e^{-9t}), \quad t \ge \frac{1}{50}.
\end{aligned}
\end{equation*}
On the other hand, when $0< t < \frac{1}{50}$, $2(e^{-t} - 3e^{-9t})<0<S_1\left(t, \frac{1}{4}\right).$
Thus we know that  when $t > 0$,
$$S_1\left(t, \frac{1}{4}\right) > 2(e^{-t} - 3e^{-9t}),$$
which implies
 $$S(t, \xi_0) > 2\left( e^{-t}-3e^{-9t} \right) \left( 1+2e^{-t} \right) ^{d-1}\ge
2e^{-t}+4\left( d-1 \right) e^{-2t}-6e^{-9t}-12\left( d-1 \right) e^{-10t}.$$
Therefore
$$\begin{aligned}
    |m_1\left( \xi _0 \right) |= \frac{1}{\pi ^{\frac{d+1}{2}}}\int_0^{\infty}{t}^{\frac{d-1}{2}}S\left( t,\xi _0 \right) \,\text{d}t \geq \frac{1}{\pi ^{\frac{d+1}{2}}} \int_0^{\infty}{t^{\frac{d-1}{2}}}\varphi \left( t \right) dt,
\end{aligned}$$
where\begin{equation*}
    \varphi \left( t \right) \,:=2e^{-t}+4\left( d-1 \right) e^{-2t}-6e^{-9t}-12\left( d-1 \right) e^{-10t}.
\end{equation*}

Then we can obtain a lower bound for $|m_1(\xi_0)|$,
$$
\begin{aligned}
	|m_1\left( \xi _0 \right) |\ge 2c_d\left( 1+\frac{\sqrt{2}\left( d-1 \right)}{2^{\frac{d}{2}}}-\frac{1}{3^d}-\frac{6\left( d-1 \right)}{10^{\frac{d+1}{2}}}\right) .\\
\end{aligned}
$$
Therefore
$$\|R_{dis}^{(k)}\|_{\ell^2(\mathbb{Z}^d)\rightarrow \ell^2(\mathbb{Z}^d)}\geq 2c_d\left( 1+\frac{\sqrt{2}\left( d-1 \right)}{2^{\frac{d}{2}}}-\frac{1}{3^d}-\frac{6\left( d-1 \right)}{10^{\frac{d+1}{2}}} \right),$$
which complete the proof of Theorem \ref{l2}.

\section{The $\ell^{1,\infty}$ and $\ell^p$ estimates}
In this section, we will prove the optimal dimension-dependent $\ell^{1,\infty}$ and $\ell^p$ estimates of the discrete Riesz transforms $R_{dis}^{(k)}$. Without lose of generality, we may assume $k = 1$. We will still adopt the ``continuous-discrete operator'' approach. Recall that
the continuous-discrete operator $\tilde{R}_{\mathrm{dis}}^{(1)}$
$$\tilde{R}_{\text{dis}}^{(1)}(F)(x)=\sum_{n \in \mathbb{Z}^d \setminus \{0\}}K_1(n)F(x-n),F \in L^p(\mathbb{R}^d),x \in \mathbb{R}^d,$$
where $K_1(n)=c_d\frac{n_1}{|n|^{d+1}}$.
For $1<p<\infty$, Ba\~{n}uelos, Kim and Kwa\'{s}nicki \cite{BKK} have proved that
\begin{equation*}
        \|R_{\mathrm{dis}}^{(1)}\|_{\ell^p(\mathbb{Z}^d) \to \ell^p(\mathbb{Z}^d)} = \|\tilde{R}_{\mathrm{dis}}^{(1)}\|_{L^p(\mathbb{R}^d) \to L^p(\mathbb{R}^d)}.
\end{equation*}
Here we can prove that the analogue for $p=1$ remains true.
\begin{lemma}\label{equ c-d}

For the discrete Riesz transforms $R_{dis}^{(1)}$ and the continuous-discrete operator $\tilde{R}_{\mathrm{dis}}^{(1)}$
there holds
\[
    \|\tilde{R}_{\mathrm{dis}}^{(1)}\|_{L^1(\mathbb R^d)\to L^{1,\infty}(\mathbb R^d)}
    =
    \|R_{dis}^{(1)}\|_{\ell^1(\mathbb Z^d)\to \ell^{1,\infty}(\mathbb Z^d)} .
\]
\end{lemma}

\begin{proof}

We first prove
\[
    \|\tilde{R}_{\mathrm{dis}}^{(1)}\|_{L^1\to L^{1,\infty}}
    \le
    \|R_{dis}^{(1)}\|_{\ell^1\to \ell^{1,\infty}} .
\]
For $F \in L^1(\mathbb R^d)$ and $x\in Q=[-1/2,1/2)^{d}$, denote
\[
    F_x(n)=F(x+n),
    \qquad n\in\mathbb Z^d .
\]
Then
\[
    \tilde{R}_{\mathrm{dis}}^{(1)}F(x+n)
    =
    \sum_{m\in\mathbb Z^d\setminus\{0\}} K_1(m)F(x+n-m)=
    R_{dis}^{(1)}(F_x)(n).
\]
For every \(\lambda>0\), we claim that
\[
\begin{aligned}
    \left|
        \left\{y\in\mathbb R^d:
        |\tilde{R}_{\mathrm{dis}}^{(1)}F(y)|>\lambda
        \right\}
    \right|
    &=
    \int_Q
        \#\left\{
            n\in\mathbb Z^d:
            |R_{dis}^{(1)}F_x(n)|>\lambda
        \right\}
    \,dx .
\end{aligned}
\]
Indeed, if we denote
\[
    E=\left\{y\in\mathbb R^d:
        | \tilde{R}_{\mathrm{dis}}^{(1)}F(y)|>\lambda
    \right\},
\]
then

\[
\begin{aligned}
    |E| &=
    \sum_{n\in\mathbb Z^d}
        \int_Q \mathbf 1_E(x+n)\,dx
    =\int_Q
        \#\{n\in\mathbb Z^d:x+n\in E\}
    \,dx ,
\end{aligned}
\]
which implies the claim.
Therefore,
\[
\begin{aligned}
    \lambda
    \left|
        \left\{y\in\mathbb R^d:
        |\tilde{R}_{\mathrm{dis}}^{(1)}F(y)|>\lambda
        \right\}
    \right|
    &=
    \lambda
    \int_Q
        \#\left\{
            n\in\mathbb Z^d:
            |R_{dis}^{(1)}F_x(n)|>\lambda
        \right\}
    \,dx                                                        \\
    &\le
    \|R_{dis}^{(1)}\|_{\ell^1\to\ell^{1,\infty}}
    \int_Q \|F_x\|_{\ell^1(\mathbb Z^d)}\,dx                    \\
    &=
    \|R_{dis}^{(1)}\|_{\ell^1\to\ell^{1,\infty}}
    \|F\|_{L^1(\mathbb R^d)} .
\end{aligned}
\]
which implies
 \[
    \|\tilde{R}_{\mathrm{dis}}^{(1)}\|_{L^1\to L^{1,\infty}}
    \le
    \|R_{dis}^{(1)}\|_{\ell^1\to \ell^{1,\infty}} .
\]
Conversely, for \(f\in\ell^1(\mathbb Z^d)\) let
\[
    F(x)
    =
    \sum_{n\in\mathbb Z^d}f(n)\mathbf 1_Q(x-n).
\]
Then $\|F\|_{L^1(\mathbb R^d)}
    =
    \|f\|_{\ell^1(\mathbb Z^d)}$
and for \(x\in n+Q\)
\[
\begin{aligned}
     \tilde{R}_{\mathrm{dis}}^{(1)}F(x)=
    \sum_{m\in\mathbb Z^d\setminus\{0\}} K_1(m)F(x-m)          =
    \sum_{m\in\mathbb Z^d\setminus\{0\}} K_1(m)f(n-m)          =
    R_{dis}^{(1)}f(n).
\end{aligned}
\]
Consequently,
\[
\begin{aligned}
    \|\tilde{R}_{\mathrm{dis}}^{(1)}F\|_{L^{1,\infty}}
    &=
    \sup_{\lambda>0}
    \lambda
    \left|
        \left\{x\in\mathbb R^d:
        |\tilde{R}_{\mathrm{dis}}^{(1)}F(x)|>\lambda
        \right\}
    \right|
    =
    \sup_{\lambda>0}
    \lambda
    \sum_{\substack{n\in\mathbb Z^d\\ |R_{dis}^{(1)}f(n)|>\lambda}}
    |Q|                                                       \\
    &=
    \sup_{\lambda>0}
    \lambda
    \#\left\{
        n\in\mathbb Z^d:
        |R_{dis}^{(1)}f(n)|>\lambda
    \right\}
    =
    \|R_{dis}^{(1)}f\|_{\ell^{1,\infty}} .
\end{aligned}
\]
It follows that
\[
\begin{aligned}
    \|R_{dis}^{(1)}f\|_{\ell^{1,\infty}}
    =
    \|\tilde{R}_{\mathrm{dis}}^{(1)}F\|_{L^{1,\infty}}
    \le
    \|\tilde{R}_{\mathrm{dis}}^{(1)}\|_{L^1\to L^{1,\infty}}
    \|F\|_{L^1(\mathbb R^d)}
    =
    \|\tilde{R}_{\mathrm{dis}}^{(1)}\|_{L^1\to L^{1,\infty}}
    \|f\|_{\ell^1(\mathbb Z^d)} ,
\end{aligned}
\]
which implies
\[
    \|\tilde{R}_{\mathrm{dis}}^{(1)}\|_{L^1(\mathbb R^d)\to L^{1,\infty}(\mathbb R^d)}
    =
    \|R_{dis}^{(1)}\|_{\ell^1(\mathbb Z^d)\to \ell^{1,\infty}(\mathbb Z^d)} .
\]
\end{proof}

Before the proof of the Theorem \ref{weak l1 and lp} we state two important Lemmas. And the first one may be of independent interest.
\begin{lemma}\label{est sum}
There is a absolute constant $\delta>0$ such that
$$\sum_{\substack{z\in\mathbb Z^d\\1\le |z|\le d^2}}
    \frac{|z_1|}{|z|^{d+1}}\leq \delta.$$
Furthermore when $d \to \infty$,
$$\sum_{\substack{z\in\mathbb Z^d\\1\le |z|\le d^2}}
    \frac{|z_1|}{|z|^{d+1}}= 2+\frac{2\left( \sqrt{2}+o\left( 1 \right) \right) d}{2^{\frac{d}{2}}} .$$
\end{lemma}
\begin{proof}
Denote
$$B_d=\sum_{\substack{z\in\mathbb Z^d\\1\le |z|\le d^2}}
    \frac{|z_1|}{|z|^{d+1}}.$$
It is obviously that
$$B_d\geq 2+\sum_{\substack{z_1=\pm 1\\ |z|= \sqrt{2}}}\frac{|z_1|}{|z|^{d+1}}=2+\frac{2\sqrt{2} (d-1)}{2^{\frac{d}{2}}}.$$ For the upper bound we will adopt the same strategy as we did in Section 2. Using the identity $\frac{1}{|z|^{d+1}}=\frac{1}{\Gamma\left(\frac{d+1}{2}\right)}
    \int_0^{\infty}t^{\frac{d-1}{2}}e^{-t|z|^2}\,dt $ we have
\begin{align*}
    B_d
     & =
    \sum_{\substack{z\in\mathbb Z^d\\1\le |z|\le d^2}}
    \frac{|z_1|}{|z|^{d+1}}
    =
    \sum_{\substack{z\in\mathbb Z^d\\1\le |z|\le d^2}} |z_1|\frac{1}{\Gamma\left(\frac{d+1}{2}\right)}
    \int_0^{\infty}t^{\frac{d-1}{2}}e^{-t|z|^2}\,dt                     \\
    & \le
    \frac{1}{\Gamma\left(\frac{d+1}{2}\right)}
    \int_0^{\infty}t^{\frac{d-1}{2}}
    \left(\sum_{z_1=-d^2}^{d^2}|z_1|e^{-tz_1^2}\right)
    \left(\sum_{k=-\infty}^{\infty}e^{-tk^2}\right)^{d-1}\,dt           \\
    & =
    \frac{2}{\Gamma\left(\frac{d+1}{2}\right)}
    \int_0^{\infty}t^{\frac{d-1}{2}}D_d(t)\Theta(t)^{d-1}\,dt,
\end{align*}
where
\[
    \Theta(t):=\sum_{k=-\infty}^{\infty}e^{-tk^2},
    \qquad
    D_d(t):=\sum_{k=1}^{d^2}ke^{-tk^2}.
\]
Set
\[
    J_{\mathrm{small}}
    =
    \frac{2}{\Gamma\left(\frac{d+1}{2}\right)}
    \int_0^1 t^{\frac{d-1}{2}}D_d(t)\Theta(t)^{d-1}\,dt,
\]
and
\[
    J_{\mathrm{large}}
    =
    \frac{2}{\Gamma\left(\frac{d+1}{2}\right)}
    \int_1^{\infty} t^{\frac{d-1}{2}}D_d(t)\Theta(t)^{d-1}\,dt.
\]

For $t\in(0,1]$, the Poisson summation formula (see Lemma \ref{lem:poisson_sums}) reveals
\[
    \Theta(t)
    =\sqrt{\frac{\pi}{t}}
    \left(1+2\sum_{n=1}^{\infty}e^{-\frac{\pi^2n^2}{t}}\right)
    \le
    4\sqrt{\frac{\pi}{t}}
    \left(1+e^{-\frac{\pi^2}{t}}\right).
\]
Consequently,
\begin{align*}
    J_{\mathrm{small}}
    \le
    2\frac{(4\sqrt{\pi})^{d-1}}{\Gamma(\frac{d+1}{2})}
    \biggl(
    \int_0^1 D_d(t)\,dt
    +
    \int_0^1
    \left[
    \left(1+e^{-\frac{\pi^2}{t}}\right)^{d-1}-1
    \right]
    D_d(t)\,dt
    \biggr).
\end{align*}
For the first term,
\[
    \int_0^1D_d(t)\,dt
    =\sum_{k=1}^{d^2}k\int_0^1e^{-tk^2}\,dt
    \le \sum_{k=1}^{d^2}k\cdot\frac{1}{k^2}
    \le 2\log(d^2).
\]
For the second term, notice that for $t\in(0,1]$
\[
    D_d(t)\leq \sum_{a=1}^{\infty}ae^{-ta^2}
    \le \left(1+\frac1{2t}\right)e^{-t}
    \le \frac{2}{t}.
\]
then we can obtain
\begin{align*}
    &\int_0^1
    \left[
    \left(1+e^{-\frac{\pi^2}{t}}\right)^{d-1}-1
    \right]D_d(t)\,dt
     \le
    2\int_0^1
    \left[
    \left(1+e^{-\frac{\pi^2}{t}}\right)^{d-1}-1
    \right]\frac{dt}{t}                                                 \\
    & =
    2\int_0^1
    \sum_{m=1}^{d-1}\binom{d-1}{m}e^{-\frac{m\pi^2}{t}}\frac{dt}{t}
    \le
    2\sum_{m=1}^{d-1}\binom{d-1}{m}e^{-m\pi^2}           \le 2(1+e^{-\pi^2})^{d-1}.
\end{align*}
Combining the
above estimates gives
\begin{align}\label{est J small}
    J_{\mathrm{small}}
    \le
    \frac{(4\sqrt{\pi})^{d-1}}{\Gamma(\frac{d+1}{2})}
    4\log(d^2)
    +
    \frac{(4\sqrt{\pi})^{d-1}}{\Gamma(\frac{d+1}{2})}
    4(1+e^{-\pi^2})^{d-1}.
\end{align}

For $J_{\mathrm{large}}$, since $ D_d(t)
    \le \sum\limits_{k=1}^{\infty}ke^{-tk^2}$ we can adopt the same estimation for $I_{\mathrm{large}}$ in section 2 and obtain
 \begin{align}\label{est J large}\nonumber
    J_{large}
&\le
2+\frac{2\sqrt2(d-1)}{2^{\frac{d}{2}}}
+\frac{6(d-1)}{5^{\frac{d+1}{2}}}+9d^2e^{\frac{3}{d}}\frac{1}
{3^{\frac{d+1}{2}}}+5e^{\frac{3}{d}}\frac{1}
{4^{\frac{d+1}{2}}}\\
&\quad +\frac{9d^2(2\ln d)^{\frac{d+1}{2}}
\left(1+\frac3e\right)^d}{\Gamma\left(\frac{d+1}{2}\right)}+\frac{5(2\ln d)^{\frac{d+1}{2}}
\left(1+\frac3e\right)^d}{\Gamma\left(\frac{d+1}{2}\right)}.
 \end{align}
By (\ref{est J small}) and (\ref{est J large}), we have
\begin{align}\label{est Bd}
    B_d=2+\frac{2\left( \sqrt{2}+o\left( 1 \right) \right) d}{2^{\frac{d}{2}}},
\end{align}
which complete the proof of the Lemma.
\end{proof}
For $z\in\mathbb Z^d$, denote
\[
    \widetilde K_1^*(z)
    :=\int_Q\int_Q
    \bigl(K_1^*(z+s-t)-K_1^*(z)\bigr)\,dt\,ds ,
\]
where
\[
    K_1^*(x)=c_d\frac{x_1}{|x|^{d+1}}\chi_{\{|x|\ge 1\}}.
\]
We provide the $\ell^1$ estimate for $\widetilde K_1^*(z)$ in the following Lemma.
\begin{lemma}\label{l1 est}
There is a absolute constant $\eta>0$ such that
\[
    \|\widetilde K_1^*\|_{\ell^1(\mathbb Z^d)}\le \eta c_d .
\]
Furthermore when $d \to \infty$,
\[
    \|\widetilde K_1^*\|_{\ell^1(\mathbb Z^d)}\leq \left(2+\frac{2\left( \sqrt{2}+o\left( 1 \right) \right) d}{2^{\frac{d}{2}}}\right) c_d .
\]
\end{lemma}

\begin{proof}

We split the sum into two parts according to the size of $z$.
First $|z|>d^2$,  by the mean value theorem and the fact $|z+\theta(s-t)|\geq (1-\frac{\sqrt{d}}{d^2})|z|$ for $s,t\in Q$ and some $0<\theta<1$, we have
\begin{align}\label{est of K1}\nonumber
    \bigl|\widetilde K_1^*(z)\bigr|
    & =
    \left|
    \int_Q\int_Q
    \bigl(K_1^*(z+s-t)-K_1^*(z)\bigr)\,dt\,ds
    \right|                                      \\\nonumber
    & \le
    \int_Q\int_Q
    |s-t|\,\bigl|\nabla K_1(z+\theta(s-t))\bigr|\,dt\,ds                  \\\nonumber
    & \le
    c_d(d+2)
    \frac{1}{\left(1-\frac{\sqrt d}{d^2}\right)^{d+1}}
    \frac{1}{|z|^{d+1}}
    \int_Q\int_Q |s-t|\,dt\,ds                                          \\
    & \leq
    2e^4c_d d^{3/2}\frac{1}{|z|^{d+1}} .
\end{align}

Next let us estimate $\sum\limits_{|z|>d^2}|z|^{-d-1}$.  First we claim that
\[
    \sum_{|z|>d^2}\frac1{|z|^{d+1}}
    \le
    e\int_{|x|>d^2/2}\frac{dx}{|x|^{d+1}}.
\]
Set $Q_z=z+\left[-\frac12,\frac12\right)^d$. If $x\in Q_z$, then $|x-z|\le \sqrt d/2$. Since $|z|>d^2$ then
\[
    |x|\le |z|\left(1+\frac{\sqrt d}{2|z|}\right)
    \le (1+\frac{1}{2d^{3/2}})|z|.
\]
It follows that for $x\in Q_z$
\[
    \frac1{|z|^{d+1}}
    \le \frac{e}{|x|^{d+1}}.
\]
Integrating over $Q_z$ gives
\[
    \frac1{|z|^{d+1}}
    \le e\int_{Q_z}\frac{dx}{|x|^{d+1}}.
\]
Summing over $|z|>d^2$, we complete the proof of the claim
\begin{equation}\label{est of sum}
    \sum_{|z|>d^2}\frac1{|z|^{d+1}}
    \le
    e\int_{\bigcup\limits_{|z|>d^2}Q_z}\frac{dx}{|x|^{d+1}}
    \le
    e\int_{|x|>d^2/2}\frac{dx}{|x|^{d+1}}.
\end{equation}
Consequently, by (\ref{est of K1}) and (\ref{est of sum}) we obtain
\begin{align}\label{est far region}\nonumber
    \sum_{|z|>d^2}\bigl|\widetilde K_1^*(z)\bigr|
    & \leq
    2e^4c_d d^{3/2}\sum_{|z|>d^2}\frac1{|z|^{d+1}}                         \\\nonumber
    & \leq
    2e^5c_d d^{3/2}|\mathbb S^{d-1}|
    \int_{d^2/2}^{\infty}\frac{r^{d-1}}{r^{d+1}}\,dr                    \\
    & =8e^5d^{-\frac{1}{2}}\frac{\pi^{\frac{d}{2}}}{\Gamma(\frac{d}{2})}c_d .
\end{align}

We now consider the complementary region $|z|\le d^2$. A change of variables reveals
\[
    \widetilde K_1^*(z)
    =\int_{[-1,1)^d}\Phi(u)
    \bigl(K_1^*(z+u)-K_1^*(z)\bigr)\,du,
\]
where
\[
    \Phi(u):=\prod_{j=1}^d(1-|u_j|)_+,
    \qquad u\in\mathbb R^d.
\]
Indeed, if $u=s-t$, then $u\in[-1,1)^d$ and $t\in Q\cap(Q-u)$.  Therefore,
for any integrable function $H$, we have
\begin{align*}
    \int_Q\int_Q H(s-t)\,ds\,dt
    & =
    \int_{[-1,1)^d}\int_{Q\cap(Q-u)}H(u)\,dt\,du                       \\
    & =
    \int_{[-1,1)^d}H(u)|Q\cap(Q-u)|\,du.
\end{align*}
Since $|Q\cap(Q-u)|=\Phi(u)$, the asserted identity follows by taking
$H(u)=K_1^*(z+u)-K_1^*(z)$. And it is easy to check the function $\Phi$ satisfies the following elementary properties
\[
    0\le \Phi\le 1,
    ~~~~
    \int_{\mathbb R^d}\Phi(u)\,du=1
    ~~~~~\text{and}~~~~
    \sum_{r\in\mathbb Z^d}\Phi(x-r)=1
    ~~~\text{for}~~~~ x\in\mathbb R^d.
\]
Hence
\[
    \bigl|\widetilde K_1^*(z)\bigr|
    \le
    \int_{\mathbb R^d}\Phi(u)|K_1^*(z+u)|\,du+|K_1^*(z)|.
\]
Summing over $|z|\le d^2$ yields
\begin{align}\label{est near region}
    \sum_{|z|\le d^2}\bigl|\widetilde K_1^*(z)\bigr|
    \le I_d+II_d,
\end{align}
where
\[
    I_d:=\sum_{|z|\le d^2}\int_{\mathbb R^d}\Phi(u)|K_1^*(z+u)|\,du,
    ~~~\text{and}~~~
    II_d:=\sum_{|z|\le d^2}|K_1^*(z)|.
\]

For $I_d$, using the fact $\sum\limits_{|z|\le d^2}\Phi(x-z)=0$ when $|x|> d^2+\sqrt d$ and $\sum\limits_{r\in\mathbb Z^d}\Phi(x-r)=1$, we can derive that
\begin{align}\label{est Id}\nonumber
    I_d
    & =
    \sum_{|z|\le d^2}\int_{\mathbb R^d}\Phi(x-z)|K_1^*(x)|\,dx=\int_{\mathbb R^d}|K_1^*(x)|
    \sum_{|z|\le d^2}\Phi(x-z)\,dx  \\\nonumber
    & \leq c_d\int_{1\le |x|\le d^2+\sqrt d}
    \frac{|x_1|}{|x|^{d+1}}\,dx =c_d\int_1^{d^2+\sqrt d}\frac{1}{\rho}
    \int_{\mathbb S^{d-1}}|\theta_1|\,d\sigma(\theta)\,d\rho   \\
     & =
    \frac{2}{\pi}\log(d^2+\sqrt d).
\end{align}
By Lemma \ref{est sum}, we know
\begin{align}\label{est IId}
II_d\leq \left(2+\frac{2\left( \sqrt{2}+o\left( 1 \right) \right) d}{2^{\frac{d}{2}}}\right)c_d.
\end{align}
Therefore, by (\ref{est far region}), (\ref{est near region}), (\ref{est Id}) and (\ref{est IId}) we have
\[
    \|\widetilde K_1^*\|_{\ell^1(\mathbb Z^d)}\leq \left(2+\frac{2\left( \sqrt{2}+o\left( 1 \right) \right) d}{2^{\frac{d}{2}}}\right) c_d,
\]
which complete the proof of the Lemma.
\end{proof}

\subsection{The $\ell^{1,\infty}$ case} Now we are in the position to prove the optimal dimension-dependent $\ell^{1,\infty}$ estimate of the discrete Riesz transforms. Since we already know $2c_d$ is a lower bound for $\|R_{dis}^{(k)}\|_{\ell^1(\mathbb{Z}^d)\rightarrow \ell^{1,\infty}(\mathbb{Z}^d)}$ in section 1, we only need to obtain the optimal upper bound.

Let \(f\in\ell^1(\mathbb Z^d)\) and
\[
    F(x)
    =
    \sum_{n\in\mathbb Z^d}f(n)\mathbf 1_Q(x-n).
\]
Thus for \(x\in n+Q\),
\[
\begin{aligned}
    \tilde{R}_{\mathrm{dis}}^{(1)}F(x)
    =
    \sum_{\substack{m\in\mathbb Z^d\\ m\neq n}}
        K_1(n-m)f(m)=\sum_{\substack{m\in\mathbb Z^d\\ m\neq n}}
        K_1^*(n-m)f(m),
\end{aligned}
\]
where
\[
    K_1^*(x)=c_d\frac{x_1}{|x|^{d+1}}\chi_{\{|x|\ge 1\}}.
\]
For the truncated  Riesz transform operator $R_1$, we have
\[
\begin{aligned}
    R_1F(x)&=\int_{\mathbb R^d}K_1^*(x-y)F(y)\,dy                             =
    \sum_{\substack{m\in\mathbb Z^d\\ m\neq n}}
        f(m)
        \int_Q K_1^*(x-m-t)\,dt .
\end{aligned}
\]

Define the error term
\[
    E(x)
    :=
    \tilde{R}_{\mathrm{dis}}^{(1)}F(x)-R_1F(x).
\]
Then by Lemma \ref{equ c-d} we obtain
\[
\begin{aligned}
    \|R_{dis}^{(1)}f\|_{\ell^{1,\infty}}
    &=
    \|\tilde{R}_{\mathrm{dis}}^{(1)}F\|_{L^{1,\infty}}       \le
    \frac{ 3^{\frac{d}{2}+2}}{3^{\frac{d}{2}+2}-d}\|E\|_{L^1}
    +
    \frac{3^{\frac{d}{2}+2}}{d}\|R_1F\|_{L^{1,\infty}}
\end{aligned}
\]
On one hand, by the fact
\[
\begin{aligned}
    \|E\|_{L^1(\mathbb R^d)}
    =
    \sum_{n\in\mathbb Z^d}\int_Q |E(n+s)|\,ds ,
\end{aligned}
\]
and
\[
\begin{aligned}
    E(n+s)
    =
    \sum_{\substack{m\in\mathbb Z^d\\ m\neq n}}
        f(m)
        \int_Q
            \bigl[
                K_1^*(n-m)-K_1^*(n+s-m-t)
            \bigr]
        \,dt ,
\end{aligned}
\]
we obtain
\[
\begin{aligned}
    \|E\|_{L^1(\mathbb R^d)}
    &\le
    \sum_{n\in\mathbb Z^d}
    \int_Q
    \sum_{\substack{m\in\mathbb Z^d\\ m\neq n}}
        |f(m)|
        \int_Q
            |K_1^*(n-m)-K_1^*(n+s-m-t)|
        \,dt\,ds                                                  \\
    &=
    \sum_{m\in\mathbb Z^d}|f(m)|
    \sum_{\substack{n\in\mathbb Z^d\\ n\neq m}}
    \int_Q\int_Q
        |K_1^*(n-m)-K_1^*(n+s-m-t)|
    \,dt\,ds                                                       \\
    &\leq
    2c_d\left(1+\frac{\left( \sqrt{2}+o\left( 1 \right) \right) d}{2^{\frac{d}{2}}}\right) \|f\|_{\ell^1(\mathbb Z^d)},
\end{aligned}
\]
where we use Lemma \ref{l1 est} to get the last equality.
On the other hand by the same proof of Janakiraman in \cite{Ja}, one can obtain that there exists $C'>0$ independent of $d$, such that
\[
\|R_1F\|_{L^{1,\infty}(\mathbb R^d)} \le C'\log d\,\|F\|_{L^1(\mathbb R^d)}.
\]

Therefore from the above estimates we conclude that
\[
    \|R_{dis}^{(1)}\|_{\ell^1\to\ell^{1,\infty}}
    =2c_d\left(1+\frac{\left( \sqrt{2}+o\left( 1 \right) \right) d}{2^{\frac{d}{2}}}\right).
\]

\subsection{The $\ell^p$ case}
Now let us prove the optimal dimension-dependent $\ell^{p}$ estimate of the discrete Riesz transforms. First we consider the lower bound, which can be derived directly from the $\ell^2-$norm of the discrete Riesz transforms. In fact since the Fourier multiplier of the continuous-discrete operator $\tilde{R}_{\mathrm{dis}}^{(1)}$
$$\tilde{R}_{\text{dis}}^{(1)}(F)(x)=\sum_{n \in \mathbb{Z}^d \setminus \{0\}}K_1(n)F(x-n),F \in L^p(\mathbb{R}^d),x \in \mathbb{R}^d,$$
is $\mathrm{p.v.} \sum\limits_{m \in \mathbb{Z}^d \setminus \{0\}}c_d\frac{m_1}{|m|^{d+1}}e^{-2\pi i m \cdot \xi}$,
then the Riesz-Thorin interpolation Theorem tell us
$$ \|\tilde{R}_{\mathrm{dis}}^{(1)}\|_{L^2(\mathbb{R}^d) \to L^2(\mathbb{R}^d)}\leq  \|\tilde{R}_{\mathrm{dis}}^{(1)}\|^{1/2}_{L^p(\mathbb{R}^d) \to L^p(\mathbb{R}^d)}
 \|\tilde{R}_{\mathrm{dis}}^{(1)}\|^{1/2}_{L^{p^{\prime}}(\mathbb{R}^d) \to L^{p^\prime}(\mathbb{R}^d)}=\|\tilde{R}_{\mathrm{dis}}^{(1)}\|_{L^p(\mathbb{R}^d) \to L^p(\mathbb{R}^d)}.$$
Thus by Theorem \ref{l2} and Lemma \ref{c-d lemma} we obtain the lower bound
$$\bigl\|R_{\mathrm{dis}}^{(k)}\bigr\|_{\ell^p(\mathbb Z^d)\to
    \ell^p(\mathbb Z^d)}\geq 2c_d\left( 1+\frac{\sqrt{2}\left( d-1 \right)}{2^{\frac{d}{2}}}-\frac{1}{3^d}-\frac{6\left( d-1 \right)}{10^{\frac{d+1}{2}}} \right).$$

For the upper bound, Ba\~{n}uelos, Kim and Kwa\'{s}nicki \cite{BKK} have proved that
$$\|R_{dis}^{(k)}\|_{\ell^p(\mathbb{Z}^d)\rightarrow \ell^p(\mathbb{Z}^d)}\leq \cot(\frac{\pi}{2p^{\ast}})+C_d,$$
and there is no more information about $C(d)$ except it is a positive constant depend on $d$. Here we can propose a more refined
estimate of this error term.

Without lose of generality, we may assume $k = 1$.
Recall that
\[
    K_1^*(x)=c_d\frac{x_1}{|x|^{d+1}}\chi_{\{|x|\ge 1\}},~~~~~~~ Q=\left[-\frac12,\frac12\right)^d.
\]
and
\[
    \widetilde K_1^*(z)
    :=\int_Q\int_Q
    \bigl(K_1^*(z+s-t)-K_1^*(z)\bigr)\,dt\,ds .
\]
By Proposition 6.1 in \cite{BKK}, there holds
\[
    \bigl\|R_{\mathrm{dis}}^{(1)}\bigr\|_{\ell^p(\mathbb Z^d)
    \to \ell^p(\mathbb Z^d)}
    \le
    \cot\left(\frac{\pi}{2p^*}\right)
    +\|\widetilde K_1^*\|_{\ell^1(\mathbb Z^d)}.
\]
Thus using Lemma \ref{l1 est} we can complete the proof.

\section{Appendix}
In this section, we will prove some lemmas which were used in our proof of the main
Theorems.
\begin{lemma}\label{multiplier}
   The Fourier multiplier of the continuous-discrete operator $\tilde{R}_{\mathrm{dis}}^{(1)}$ is
    \begin{equation*}
        m_1(\xi) = \mathrm{p.v.} \sum_{n \in \mathbb{Z}^d \setminus \{0\}} K_1(n)e^{-2\pi i n \cdot \xi}.
    \end{equation*}
\end{lemma}

\begin{proof}
It suffices to prove that for every
\(\phi\in C_{c}^{\infty}([0,1]^d)\),
\[
    \left\langle\widehat{\tilde{R}_{\mathrm{dis}}^{(1)}F},\phi\right\rangle
   =
    \left\langle
        m_1\widehat F,
        \phi
    \right\rangle .
\]
Define the truncated ``continuous-discrete'' operator $\tilde{R}_{\mathrm{dis},N}^{(1)}$:
\begin{equation*}
    \tilde{R}_{\mathrm{dis},N}^{(1)}F(x) = \sum_{0 < |m| \le N} K_1(m)F(x-m), \quad F \in \mathcal{S}(\mathbb{R}^d).
\end{equation*}
then
\begin{equation*}
    \widehat{\tilde{R}_{\mathrm{dis},N}^{(1)}F}(\xi) = m_{1,N}(\xi)\widehat{F}(\xi),
\end{equation*}
where
\begin{equation*}
    m_{1,N}(\xi) = \sum_{0 < |m| \le N} K_1(m)e^{-2\pi i m \cdot \xi}.
\end{equation*}
Then have
\[
\begin{aligned}
    \left\langle m_{1,N}\widehat F,\phi\right\rangle
    -
    \left\langle m_1\widehat F,\phi\right\rangle
   &=
    \int_{[0,1]^d}
        (m_{1,N}(\xi)-m_1(\xi))
        \widehat F(\xi)\phi(\xi)
    \,d\xi \\
    & \leq \|m_{1,N}-m_1\|_{L^2([0,1]^d)}
    \|\widehat F\,\phi\|_{L^2([0,1]^d)}.
\end{aligned}
\]
Since \(K_1\in \ell^2(\mathbb Z^d\setminus\{0\})\), there holds
\[
    \|m_{1,N}-m_1\|_{L^2([0,1]^d)}^2
    =
    \sum_{|m|>N} |K(m)|^2
    \longrightarrow 0,
    \qquad\text{as } N\to\infty .
\]
Hence
\begin{align*}
    \lim_{N\to\infty}
    \left\langle m_{1,N}\widehat F,\phi\right\rangle
    =
    \left\langle m_1\widehat F,\phi\right\rangle .
\end{align*}

On the other hand we can derive that
\[
\begin{aligned}
    &\left\langle\widehat{\tilde{R}_{\mathrm{dis},N}^{(1)}F},\phi\right\rangle-\left\langle\widehat{\tilde{R}_{\mathrm{dis}}^{(1)}F},\phi\right\rangle
    =
    \int_{[0,1]^d}
        \left(
            \sum_{|m|>N} K(m)F(x-m)
        \right)
        \widehat{\phi}(x)
    \,dx \\
    &\le
    \int_{[0,1]^d}
        \left(
            \sum_{|m|>N} |K(m)|^2
        \right)^{1/2}
        \left(
            \sum_{|m|>N} |F(x-m)|^2
        \right)^{1/2}
        |\widehat{\phi}(x)|
    \,dx   \\
    &\le
    \|F\|_{L^2(\mathbb R^d)}
    \|\phi\|_{L^2([0,1]^d)}
    \left(
        \sum_{|m|>N} |K(m)|^2
    \right)^{1/2}
    \longrightarrow 0,~~~N\to\infty        .
\end{aligned}
\]
Therefore for every
\(\phi\in C_{c}^{\infty}([0,1]^d)\),
\[
    \left\langle\widehat{\tilde{R}_{\mathrm{dis}}^{(1)}F},\phi\right\rangle
   =
    \left\langle
        m_1\widehat F,
        \phi
    \right\rangle ,
\]
which complete the proof.

\end{proof}

\begin{lemma}\label{form of m1}
Let \(\xi\in [0,1]^d\). For \(m\in \mathbb Z^d\setminus\{0\}\) and \(\varepsilon>0\),
then
\[
\sum_{m\neq0}
m_1\sin(2\pi m\cdot \xi)
\int_0^\infty t^{\frac{d-1}{2}}e^{-t|m|^2}\,dt
=
\lim_{\varepsilon\to 0}
\sum_{m\neq0}
m_1\sin(2\pi m\cdot \xi)
\int_\varepsilon^\infty t^{\frac{d-1}{2}}e^{-t|m|^2}\,dt
\]
in \(L^2(\mathbb [0,1]^d)\). Moreover,
\[
\sum_{m\neq0}
m_1\sin(2\pi m\cdot \xi)
\int_0^\infty t^{\frac{d-1}{2}}e^{-t|m|^2}\,dt
=
\int_0^\infty
t^{\frac{d-1}{2}}
\sum_{m\neq0}
m_1e^{-t|m|^2}\sin(2\pi m\cdot \xi)\,dt
\]
in $[0,1]^d$.
\end{lemma}

\begin{proof}
Denote
\[
    b_m^\varepsilon
    =
    m_1\int_\varepsilon^\infty t^{\frac{d-1}{2}}e^{-t|m|^2}\,dt,
    \qquad
    b_m
    =
    m_1\int_0^\infty t^{\frac{d-1}{2}}e^{-t|m|^2}\,dt,
\]
we need to show that
\[
\sum_{m\neq0}
b_m^\varepsilon\sin(2\pi m\cdot \xi)
\to
\sum_{m\neq0}
b_m\sin(2\pi m\cdot \xi)
\]
in \(L^2(\mathbb T^d)\).

Ii is easy to check that
$b_m
    =
    \Gamma\left(\frac{d+1}{2}\right)\frac{m_1}{|m|^{d+1}}$
and $\{b_m\}\in \ell^2(\mathbb Z^d\setminus\{0\})$.
For every fixed \(m\neq0\), we clearly have
\[
    b_m^\varepsilon\to b_m,
    ~~~~\varepsilon\to 0~~~\text{and}~~~|b_m^\varepsilon-b_m|\leq |b_m|.
\]
Then the Lebesgue dominated convergence theorem implies
\[
    \sum_{m\neq0}|b_m^\varepsilon-b_m|^2\to0,
    \qquad \varepsilon\to 0.
\]
Thus by Parseval's identity, it follows that when $\varepsilon\to0$,
\[
\begin{aligned}
\left\|
\sum_{m\neq0}
(b_m^\varepsilon-b_m)\sin(2\pi m\cdot \xi)
\right\|_{L^2(\mathbb T^d)}^2
 &=
\left\|
-i\sum_{m\neq0}
(b_m^\varepsilon-b_m)e^{2\pi i m\cdot \xi}
\right\|_{L^2(\mathbb T^d)}^2  \\
&\qquad =
\sum_{m\neq0}|b_m^\varepsilon-b_m|^2
\to0.
\end{aligned}
\]
which is
\[
\sum_{m\neq0}
b_m^\varepsilon\sin(2\pi m\cdot \xi)
\to
\sum_{m\neq0}
b_m\sin(2\pi m\cdot \xi),\varepsilon\to0
\]
in \(L^2([0,1]^d)\).

For each fixed \(\varepsilon>0\). Let
\[
    F(m,t)
    =
    t^{\frac{d-1}{2}}m_1e^{-t|m|^2}\sin(2\pi m\cdot \xi),
    \quad
    m\in\mathbb Z^d\setminus\{0\},\quad t\in(\varepsilon,\infty).
\]
It is easy to check that
\[
    \sum_{m\neq0}\int_\varepsilon^\infty |F(m,t)|\,dt<\infty.
\]
Then by Fubini's theorem, there holds
\[
\sum_{m\neq0}
\int_\varepsilon^\infty F(m,t)\,dt
=
\int_\varepsilon^\infty
\sum_{m\neq0}F(m,t)\,dt,
\]
which implies
\[
\sum_{m\neq0}
m_1\sin(2\pi m\cdot \xi)
\int_\varepsilon^\infty t^{\frac{d-1}{2}}e^{-t|m|^2}\,dt
=
\int_\varepsilon^\infty
t^{\frac{d-1}{2}}
\sum_{m\neq0}
m_1e^{-t|m|^2}\sin(2\pi m\cdot \xi)\,dt.
\]
Combining this identity with the previous \(L^2\)-convergence gives
\begin{align}\label{euqality for multiplier}
\sum_{m\neq0}
m_1\sin(2\pi m\cdot \xi)
\int_0^\infty t^{\frac{d-1}{2}}e^{-t|m|^2}\,dt
=
\lim_{\varepsilon\to 0}
\int_\varepsilon^\infty
t^{\frac{d-1}{2}}
\sum_{m\neq0}
m_1e^{-t|m|^2}\sin(2\pi m\cdot \xi)\,dt
\end{align}
a.e.~~in $[0,1]^d$.
One can check that both side of the \ref{euqality for multiplier} are continuous function, which complete the proof of the Lemma.
\end{proof}

\begin{lemma}\label{lem:poisson_sums}
Let $t > 0$ and $x\in [0,1]$. Consider the series
\begin{align}
    S_1(t, x) &= \sum_{m_1 \neq 0} m_1 e^{-t m_1^2} \sin(2\pi m_1 x), \label{eq:S1_def} \\
    S_2(t, x) &= \sum_{m_2 \in \mathbb{Z}} e^{-t m_2^2} \cos(2\pi m_2 x). \label{eq:S2_def}
\end{align}
By the Poisson summation formula, these series admit the following equivalent dual representations
\begin{align}
    S_1(t, x) &= \frac{\pi^{3/2}}{t^{3/2}} \sum_{k \in \mathbb{Z}} (x - k) e^{-\frac{\pi^2 (x-k)^2}{t}}, \label{eq:S1_dual} \\
    S_2(t, x) &= \sqrt{\frac{\pi}{t}} \sum_{k \in \mathbb{Z}} e^{-\frac{\pi^2 (x-k)^2}{t}}. \label{eq:S2_dual}
\end{align}
\end{lemma}

\begin{proof}
We first establish the identity \eqref{eq:S2_dual} for $S_2(t, x)$. Rewrite $S_2(t, x)$ as
\begin{equation*}
    S_2(t, x) = \sum_{m_2 \in \mathbb{Z}} e^{-t m_2^2 + 2\pi i m_2 x}.
\end{equation*}
Let $f(y) = e^{-t y^2}$ then its Fourier transform is given by
\begin{equation*}
    \hat{f}(\xi) = \int_{\mathbb{R}} e^{-t y^2 - 2\pi i \xi y} \, dy = \sqrt{\frac{\pi}{t}} e^{-\frac{\pi^2 \xi^2}{t}}.
\end{equation*}
Applying the Poisson summation formula $\sum_{m \in \mathbb{Z}} f(m) e^{2\pi i m x} = \sum_{k \in \mathbb{Z}} \hat{f}(k - x)$, we immediately obtain
\begin{equation*}
    S_2(t, x) = \sqrt{\frac{\pi}{t}} \sum_{k \in \mathbb{Z}} e^{-\frac{\pi^2 (k-x)^2}{t}} = \sqrt{\frac{\pi}{t}} \sum_{k \in \mathbb{Z}} e^{-\frac{\pi^2 (x-k)^2}{t}}.
\end{equation*}

To derive the identity \eqref{eq:S1_dual} for $S_1(t, x)$, we use the fact
\begin{equation*}
    \frac{\partial}{\partial x} S_2(t, x) = -2\pi \sum_{m_2 \in \mathbb{Z}} m_2 e^{-t m_2^2} \sin(2\pi m_2 x) = -2\pi S_1(t, x),
\end{equation*}
and
\begin{equation*}
    \frac{\partial}{\partial x} S_2(t, x) = \sqrt{\frac{\pi}{t}} \sum_{k \in \mathbb{Z}} e^{-\frac{\pi^2 (x-k)^2}{t}} \left( -\frac{2\pi^2 (x-k)}{t} \right).
\end{equation*}
Therefore
\begin{equation*}
    S_1(t, x) = \frac{\pi}{t} \sqrt{\frac{\pi}{t}} \sum_{k \in \mathbb{Z}} (x - k) e^{-\frac{\pi^2 (x-k)^2}{t}} = \frac{\pi^{3/2}}{t^{3/2}} \sum_{k \in \mathbb{Z}} (x - k) e^{-\frac{\pi^2 (x-k)^2}{t}}.
\end{equation*}
\end{proof}

\bibliographystyle{amsalpha}

\end{document}